\documentclass[a4paper,12pt]{amsart}
\usepackage{amssymb}

\theoremstyle{plain}
   \newtheorem{thm}{Theorem}[section]
   \newtheorem{lem}[thm]{Lemma}
   
   \newtheorem{cor}[thm]{Corollary}
   \newtheorem{prop}[thm]{Proposition}

\theoremstyle{definition}
   \newtheorem{defn}[thm]{Definition}

\theoremstyle{remark}
   \newtheorem{rem}[thm]{Remark}

\numberwithin{equation}{section}

\DeclareMathOperator{\vol}{vol}
\DeclareMathOperator{\Ric}{Ric}
\DeclareMathOperator{\BG}{BG}
\DeclareMathOperator{\CD}{CD}
\DeclareMathOperator{\Hess}{Hess}
\DeclareMathOperator{\divv}{div}
\DeclareMathOperator{\dbl}{dbl}

\DeclareMathOperator{\supp}{supp}
\DeclareMathOperator{\Lip}{Lip}
\DeclareMathOperator{\loc}{loc}

\newcommand{\Hm}{\mathcal{H}}
\newcommand{\Cut}{\text{\rm Cut}}
\newcommand{\der}[2]{\frac{d #1}{d #2}}

\newcommand{\field}[1]{\mathbb{#1}}

\newcommand{\R}{\field{R}}

\begin{document}

\title{A topological splitting theorem\\ for weighted Alexandrov spaces}


\begin{abstract}
   Under an infinitesimal version of the Bishop-Gromov relative volume
   comparison condition for a measure on an Alexandrov space,
   we prove a topological splitting theorem of Cheeger-Gromoll type.
   As a corollary, we prove an isometric splitting theorem
   for Riemannian manifolds with singularities of nonnegative
   (Bakry-Emery) Ricci curvature.
\end{abstract}

\author{Kazuhiro Kuwae}
\address{Department of Mathematics and Engineering\endgraf
   Graduate School of Science and Technology\endgraf
   Kumamoto University\endgraf
   Kumamoto, 860-8555\endgraf
   JAPAN}
\email{kuwae@gpo.kumamoto-u.ac.jp}

\author{Takashi Shioya}
\address{Mathematical Institute\endgraf
   Tohoku University\endgraf
   Sendai 980-8578\endgraf
   JAPAN}
\email{shioya@math.tohoku.ac.jp}




\thanks{The authors are partially supported by a Grant-in-Aid
   for Scientific Research No.~19540220 and 20540058 from
   the Japan Society for the Promotion of Science}

\maketitle

\section{Introduction} \label{sec:intro}

A main purpose of this paper is to prove a splitting theorem
of Cheeger-Gromoll type for singular spaces.
Since it is impossible to define the Ricci curvature tensor
on Alexandrov spaces,
we consider an infinitesimal version of the Bishop-Gromov volume
comparison condition as a candidate of the conditions of the Ricci
curvature bounded below.
Under the volume comparison condition for a measure
on an Alexandrov space,
we prove a topological splitting theorem.
As a corollary, we prove an isometric splitting theorem
for an Alexandrov space whose regular part is a smooth Riemannian
manifold of nonnegative (Bakry-Emery) Ricci curvature.

Let us present the volume comparison condition.
For a real number $\kappa$, we set
\[
s_\kappa(r) :=
\begin{cases}
   \sin(\sqrt{\kappa}r)/\sqrt{\kappa}  &\text{if $\kappa > 0$},\\
   r &\text{if $\kappa = 0$},\\
   \sinh(\sqrt{|\kappa|}r)/\sqrt{|\kappa|} &\text{if $\kappa < 0$}.
\end{cases}
\]
The function $s_\kappa$ is the solution of the Jacobi equation
$s_\kappa''(r) + \kappa s_\kappa(r) = 0$ with initial condition
$s_\kappa(0) = 0$, $s_\kappa'(0) = 1$.

Let $M$ be an Alexandrov space of curvature bounded from below locally
and set $r_p(x) := d(p,x)$ for $p,x \in M$,
where $d$ is the distance function.
For $p \in M$ and $0 < t \le 1$,
we define a subset $W_{p,t} \subset M$ and a map
$\Phi_{p,t} : W_{p,t} \to M$ as follows.
We first set $\Phi_{p,t}(p) := p \in W_{p,t}$.
A point $x$ ($\neq p$) belongs to $W_{p,t}$
if and only if there exists $y \in M$ such that
$x \in py$ and $r_p(x) : r_p(y) = t:1$, where
$py$ is a minimal geodesic from $p$ to $y$.
Since a geodesic does not branch on an Alexandrov space,
for a given point $x \in W_{p,t}$ such a point $y$ is unique and we set
$\Phi_{p,t}(x) := y$.
The triangle comparison condition
implies the local Lipschitz continuity of the map
$\Phi_{p,t} : W_{p,t} \to M$.
We call $\Phi_{p,t}$ the \emph{radial expansion map}.

Let $\mu$ be a positive Radon measure on $M$ with full support,
$N \ge 1$ a real number, and $\Omega \subset M$ a subset.
The following is an infinitesimal version of
the Bishop-Gromov volume comparison condition for $\mu$
corresponding to the condition of the lower Ricci curvature bound
$\Ric \ge (N-1)\kappa$ with dimension $N$.

\medskip
{\bf Infinitesimal Bishop-Gromov Condition $\BG(\kappa,N)$ for $\mu$
   on $\Omega$:}
\ For any $p \in M$, $t \in (\,0,1\,]$, and any measurable function
$f : M \to [\,0,+\infty\,)$ with the property ($*$) below,
we have
\[
\int_{W_{p,t}} f\circ\Phi_{p,t}(y)\;d\mu(y)
\ge \int_M \frac{t\,s_\kappa(t\,r_p(x))^{N-1}}{s_\kappa(r_p(x))^{N-1}}\,
f(x)\;d\mu(x).
\]
\begin{itemize}
\item[($*$)] $f$ has a compact support in $\Omega \setminus \{p\}$ and,
   if $\kappa > 0$, the support is contained in
   the open metric ball $B(p,\pi/\sqrt{\kappa})$ centered at $p$ of
   radius $\pi/\sqrt{\kappa}$.
\end{itemize}
We say that $\mu$ satisfies $\BG(\kappa,N)$ if it satisfies
$\BG(\kappa,N)$ on $\Omega = M$.

\medskip
For an $n$-dimensional complete
Riemannian manifold, the Riemannian volume measure satisfies $\BG 
(\kappa,n)$
if and only if the Ricci curvature satisfies $\Ric \ge (n-1)\kappa$
(see \cite[Theorem 3.2]{Ota:mcp} for the `only if' part).
We see some studies on similar (or same) conditions to $\BG(\kappa,N) 
$ in
\cite{CC:strRicI,Gr:greenbook,St:heat,KwSy:geneMCP,KwSy:sobmet,
Rm:Poincare,Ota:mcp,Wt:localcut,KwSy:BG-Alex} etc.
$\BG(\kappa,N)$ is sometimes called the Measure Contraction Property
and is weaker than the curvature-dimension (or lower $N$-Ricci  
curvature)
condition $\CD((N-1)\kappa,N)$ introduced
by Sturm \cite{St:geom} and Lott-Villani \cite{LV:Ricmm}
in terms of mass transportation.
For a measure on an Alexandrov space, $\BG(\kappa,N)$ is equivalent to
the $(\kappa/(N-1),N)$-measure contraction property introduced
by Ohta \cite{Ota:mcp}.
For an $n$-dimensional Alexandrov space of curvature $\ge \kappa$,
the $n$-dimensional Hausdorff measure $\Hm^n$ on $M$
satisfies $\BG(\kappa,n)$ (see \cite{KwSy:BG-Alex}).
Note that we do not necessarily assume $M$ to be of curvature
\emph{uniformly} bounded below.
We assume the Alexandrov curvature condition just for the local  
regularity
of the space.
If an Alexandrov space $M$ has a measure $\mu$
satisfying $\BG(\kappa,N)$, then
the dimension of $M$ is less than or equal to $N$
\cite[Corollary 2.7]{Ota:mcp}.
The infinitesimal Bishop-Gromov condition is stable under
the measured Gromov-Hausdorff convergence
(cf.~\cite[Appendix 2]{CC:strRicI} and \cite[\S 5. I$_+$] 
{Gr:greenbook}).

One of our main theorems is stated as follows.

\begin{thm}[Topological Splitting Theorem] \label{thm:splitting} Let
   $M$ be an Alexandrov space of curvature bounded from below locally
   and $\mu$ a positive Radon measure on $M$ with full support.
   Assume that, for any
   relatively compact open subset $\Omega \subset M$,
   there exists a real number $N_\Omega \ge 1$
   such that $\mu$ satisfies $\BG(0,N_\Omega)$ on $\Omega$.
   If, in addition, $M$ contains a straight line,
   then $M$ is homeomorphic to $M' \times \R$ for some metric space  
$M'$.
\end{thm}

Note that $\BG(0,N_\Omega)$ in this theorem can be replaced with
the curvature dimension condition.

This theorem is new even if $M$ is a complete Riemannian manifold.
We do not know if the isometric splitting in the theorem is true,
i.e., if $M$ is isometric to $M' \times \R$ for some Alexandrov space
$M'$, even in the case where $\mu$ is the $n$-dimensional Hausdorff  
measure.
If we replace `$\BG(0,N_\Omega)$' with `curvature $\ge 0$', then the  
isometric
splitting was proved by Milka \cite{Mk:line}, Grove-Petersen
\cite{GP:excess}
and Yamaguchi \cite{Ym:collapsing},
as a generalization of the well-known Toponogov splitting theorem.
For $n$-dimensional Riemannian manifolds with Riemannian volume measure,
$\BG(0,n)$ is equivalent to
$\Ric \ge 0$
and the isometric splitting under $\Ric \ge 0$ was proved by
Cheeger-Gromoll \cite{CG:split}.
In our case, we do not have the Weitzenb\"ock formula, so that
we cannot obtain the isometric splitting at present.

A rough idea of our proof came from that of Cheeger-Gromoll
\cite{CG:split}.
One of essential points in our proof is to prove
a generalized version of the Laplacian comparison theorem
(Theorem \ref{thm:LapComp}), where our discussion
is much different from the Riemannian case.
We also prove the maximum principle for subharmonic functions,
by using the result of the first named author \cite{Kw:maxprinsemi}
and Cheeger's theory \cite{Ch:metmeas}.

If the metric of $M$ has enough smooth part,
we can prove the isometric splitting.  For that, we consider the  
following.

\begin{defn}[Singular Riemannian space]
   $M$ is called a \emph{singular Riemannian space} if
   the following (1),(2) and (3) are satisfied.
   \begin{enumerate}
   \item $M$ is an Alexandrov space of curvature bounded below locally.
   \item The set $S_M$ of singular points is a closed set in $M$.
   \item The set $M \setminus S_M$ of non-singular points is an
     {\rm(}incomplete{\rm)} $C^2$ Riemannian manifold.
   \end{enumerate}
\end{defn}

Note that any complete Riemannian orbifold is a singular
Riemannian space.

\begin{cor} \label{cor:splitting}
   Let $M$ be an $n$-dimensional singular Riemannian space.
   If the Ricci curvature satisfies $\Ric \ge 0$ on $M \setminus S_M$,
   then $M$ is isometric to $M' \times \R^k$,
   where $M'$ is a singular Riemannian space containing no straight line
   and $k := n-\dim M'$.
\end{cor}

If $M$ is a complete Riemannian orbifold, then Corollary
\ref{cor:splitting} was proved by Borzellino-Zhu \cite{BZ:splitting}.

We next consider the Bakry-Emery Ricci curvature.
Let $n$ be an integer with $n \ge 1$,
and $N$ a real number with $N > n$, or $N = +\infty$.
On an $n$-dimensional $C^2$ Riemannian manifold
with a measure $d\mu(x) = e^{-V(x)}\,d\vol(x)$,
where $V$ is a $C^2$ function and $\vol$ the Riemannian volume
measure,
the \emph{$N$-dimensional Bakry-Emery Ricci curvature tensor
$\Ric_{N,\mu}$} is defined by
\[
\Ric_{N,\mu} :=
\begin{cases}
   \Ric + \Hess V - (N-n)^{-1} dV\otimes dV
   & \text{if $n < N < +\infty$},\\
   \Ric + \Hess V
   & \text{if $N = +\infty$}.
\end{cases}
\]

\begin{cor} \label{cor:splitting-BE}
   Let $M$ be an $n$-dimensional singular Riemannian space,
   $N$ a number with $n < N \le +\infty$,
   and $V : M \to \R$ a continuous function which is
   of $C^2$ on $M \setminus S_M$.
   We assume that $\sup_M V < +\infty$ if $N = +\infty$.
   If the Bakry-Emery Ricci curvature satisfies $\Ric_{N,\mu} \ge 0$
   on $M \setminus S_M$ for $d\mu(x) := e^{-V(x)}\,d\vol(x)$,
   then $M$ is isometric to $M' \times \R^k$ and $V$ is constant
   on $\{x\} \times \R^k$ for each $x \in M'$,
   where $M'$ is a singular Riemannian space
   containing no straight line and $k := n-\dim M'$.
\end{cor}

Corollary \ref{cor:splitting-BE} is an extension of the result of
Lichnerowicz \cite{Lch:BE}
(see also \cite{WW:comp-BE} and \cite{FLZ:splitting})
for complete Riemannian manifolds.
In the case where $N = +\infty$, the assumption $\sup_M V < +\infty$
is necessary as was pointed out by Lott \cite{Lt:propBE}
(see also \cite{WW:comp-BE}).

\begin{rem}
   (1) All the results in this paper are true even in the case where
   $M$ has non-empty boundary.  We implicitly assume the Neumann
   boundary condition on the boundary of $M$ when we consider the
   Laplacian on $M$.  In particular, the results hold for any convex
   subset of $M$.

   (2) We can apply Corollaries \ref{cor:splitting} and
   \ref{cor:splitting-BE} to get some results for the fundamental group
   of a singular Riemannian space of nonnegative (Bakry-Emery) Ricci
   curvature (cf.~\cite{CG:split}).  However, we do not know if we can
   obtain the same results for an Alexandrov space satisfying the
   infinitesimal Bishop-Gromov condition.  One of the problems is that
   we cannot prove that a covering space inherits the infinitesimal
   Bishop-Gromov condition.  Another problem is that the splitting is
   only homeomorphic.  If the space splits as $M' \times \R$
   homeomorphically, then we do not know if $M'$ is an Alexandrov space
   or not, and we cannot apply our splitting theorem to $M'$.  This is
   not enough to investigate the fundamental group.

   (3) In our previous paper \cite{KwSy:lapcomp}, we proved a Laplacian
   comparison theorem and a splitting theorem weaker than those in this
   paper.  The proof in this paper is much easier than that in
   \cite{KwSy:lapcomp}.  We will not publish \cite{KwSy:lapcomp} from
   any journal.

   (4) Recently, H.-C.~Zhang and X.-P.~Zhu \cite{ZZ:splitting} have
   proved a version of an isometric splitting theorem under a new
   condition corresponding to the nonnegativity of Ricci curvature.
   Their condition implies the curvature-dimension condition and the
   infinitesimal Bishop-Gromov condition for the Hausdorff measure.
\end{rem}

\section{Preliminaries} \label{sec:prelim}

A \emph{geodesic space} is defined to be a metric space in which
any two points $x$ and $y$ can be joined by a length-minimizing curve
whose length is equal to the distance between $x$ and $y$.
Let $M$ be a proper geodesic space, where
`\emph{proper}' means that any bounded subset of $M$ is relatively  
compact.
We call a locally (resp.~globally) length-minimizing curve
in $M$ a \emph{geodesic} (resp.~a \emph{minimal geodesic}).
Denote by $M^2(\kappa)$ a complete simply connected $2$-dimensional  
space form
of constant curvature $\kappa$.
For three different points $x,y,z \in M$ and a real number $\kappa$,
we denote by $\tilde\angle_\kappa xyz$ the angle between
a minimal geodesic from $\tilde y$ to $\tilde x$ and
a minimal geodesic from $\tilde y$ to $\tilde z$
for three points $\tilde x, \tilde y, \tilde z \in M^2(\kappa)$
such that $d(\tilde x,\tilde y) = d(x,y)$,
$d(\tilde y,\tilde z) = d(y,z)$ and
$d(\tilde z,\tilde x) = d(z,x)$,
where $d$ is the distance function.
$\tilde\angle_\kappa xyz$ is uniquely determined if either the  
following (1)
or (2) is satisfied.
\begin{enumerate}
\item $\kappa \le 0$.
\item $\kappa > 0$ and $d(x,y) + d(y,z) + d(z,x) < \pi/\sqrt{\kappa}$.
\end{enumerate}
A proper geodesic space $M$ is said to be an \emph{Alexandrov space}
(\emph{of curvature bounded below locally}) if
for any point $x \in M$ there exists a neighborhood $U$ of $x$
and a real number $\kappa$ such that,
for any different four points $p,q_1,q_2,q_3 \in U$, we have
\begin{itemize}
\item[(T)] $\tilde\angle_\kappa q_ipq_j$, $i,j=1,2,3$, are all  
defined and satisfy
   \[
   \tilde\angle_\kappa q_1pq_2 + \tilde\angle_\kappa q_2pq_3
   + \tilde\angle_\kappa q_3pq_1 \le 2\pi.
   \]
\end{itemize}
For a given point $x \in M$, we denote by $\underline{\kappa}(x)$
the supremum of such $\kappa$'s.
Then, $\underline{\kappa}(x)$ is lower semi-continuous in $x \in M$,
so that $\underline{\kappa}$ is bounded from below on any
compact subset of an Alexandrov space.
The globalization theorem states that, for any compact subset $\Omega 
$ of
an Alexandrov space $M$, there exists a compact set $\Omega' \supset  
\Omega$
such that (T) holds for any mutually different
$p,q_1,q_2,q_3 \in \Omega$ and for any real number $\kappa$
with $\kappa \le \inf_{x\in\Omega'}\underline{\kappa}(x)$,
provided that $M$ is not a $1$-dimensional Riemannian manifold.
For a constant $\kappa$, we say that $M$ is of
\emph{curvature $\ge \kappa$}
if (T) holds for any mutually different four points $p,q_1,q_2,q_3  
\in M$.
In the case where $M$ is not a $1$-dimensional Riemannian manifold,
the globalization theorem implies that
$M$ is of curvature $\ge \kappa$ if and only if
$\underline{\kappa} \ge \kappa$ on $M$.
For a $1$-dimensional complete Riemannian manifold $M$ and for
$\kappa > 0$, $M$ is of curvature $\ge \kappa$ if and only if
the diameter of $M$ is at most $\pi/\sqrt{\kappa}$, i.e.,
$M$ is isometric to either a segment of length at most $\pi/\sqrt 
{\kappa}$
or a circle of length at most $2\pi/\sqrt{\kappa}$.

In this paper, we always assume that all Alexandrov spaces have
finite Hausdorff dimensions.
Refer to \cite{BGP,OS:rstralex,KMS:lap} for the basics
of the geometry and analysis on Alexandrov spaces, such as,
the space of directions, the tangent cone, etc.

Let $M$ be an Alexandrov space of Hausdorff dimension $n < +\infty$.
Then, $n$ coincides with the covering dimension of $M$, which is a
nonnegative integer.
Take any point $p \in M$ and fix it.
Denote by $\Sigma_pM$ the space of directions at $p$,
and by $K_pM$ the tangent cone at $p$.
$\Sigma_pM$ is an $(n-1)$-dimensional compact Alexandrov space
of curvature $\ge 1$ and
$K_pM$ an $n$-dimensional Alexandrov space of curvature $\ge 0$.

\begin{defn}[Singular point, $\delta$-singular point]
   A point $p \in M$ is called a \emph{singular point of $M$}
   if $\Sigma_pM$ is not isometric to the unit sphere $S^{n-1}$.
   For $\delta > 0$, we say that a point $p \in M$ is
   \emph{$\delta$-singular}
   if $\Hm^{n-1}(\Sigma_pM) \le \vol(S^{n-1})-\delta$.
   Let us denote the set of singular points of $M$ by
   $S_M$ and the set of $\delta$-singular points of $M$ by
   $S_\delta$.
\end{defn}

Note that a point $p \in M$ is non-singular if and only if
the tangent cone $K_pM$ is isometric to $\R^n$.
We have $S_M = \bigcup_{\delta > 0} S_\delta$.
Since the map $M \ni p \mapsto \Hm^n(\Sigma_pM)$ is lower
semi-continuous,
the set $S_\delta$ of $\delta$-singular points in $M$ is a closed set.
The following lemma is sometimes useful.

\begin{lem}[cf.~\cite{Pt:parallel}] \label{lem:geod}
   Let $\gamma$ be a minimal geodesic joining two points $p$ and $q$  
in $M$.
   Then, the spaces of directions $\Sigma_xM$ at all points
   $x \in \gamma \setminus \{p,q\}$ are isometric to each other.
   In particular, any minimal geodesic joining two non-singular
   {\rm(}resp.~non-$\delta$-singular{\rm)} points is contained in the  
set of
   non-singular {\rm(}resp.~non-$\delta$-singular{\rm)} points
   {\rm(}for any $\delta > 0${\rm)}.
\end{lem}

\begin{defn}[Boundary]
   The boundary of an Alexandrov space $M$ is defined inductively.
   If $M$ is one-dimensional, then $M$ is a complete Riemannian manifold
   and the \emph{boundary of $M$} is defined as usual.
   Assume that $M$ has dimension $\ge 2$.
   A point $p \in M$ is said to be a \emph{boundary point of $M$} if
   $\Sigma_pM$ has non-empty boundary.
\end{defn}

Any boundary point of $M$ is a singular point.
More strongly, the boundary of $M$ is contained in
$S_\delta$ for a sufficiently small $\delta > 0$, which follows from
the Morse theory in \cite{Pr:alex2,Pr:morse}.

The doubling theorem (cf.~\cite[\S 5]{Pr:alex2}, \cite[13.2]{BGP})
states that if $M$ has non-empty boundary, then the \emph{double of $M$}
(i.e., the gluing of two copies of $M$ along their boundaries)
is an Alexandrov space without boundary and each copy of $M$ is
convex in the double.

Denote by $\hat S_M$ (resp.~$\hat S_\delta$)
the set of singular (resp.~$\delta$-singular) points of
$\dbl(M)$ contained in $M$, where $M$ is identified with a copy in
$\dbl(M)$.
We agree that $\hat S_M = S_M$ and $\hat S_\delta = S_\delta$ provided
$M$ has no boundary.

The following shows the existence of differentiable and
Riemannian structure on $M$.

\begin{thm}
   \label{thm:str}
   For an $n$-dimensional Alexandrov space $M$, we have the
   following:

   {\rm(1)} There exists a number $\delta_n > 0$ depending only on $n$
     such that
     $M^* := M \setminus \hat S_{\delta_n}$ is a manifold {\rm(}with
     boundary $\partial M^*${\rm)}
     \cite{BGP,Pr:alex2,Pr:morse}
     and has a natural $C^\infty$ differentiable structure
     {\rm(}even on the boundary $\partial M^*${\rm)}
     \cite{KMS:lap}.

   {\rm(2)} The Hausdorff dimension of $S_M$ is at most $n-1$
     \cite{BGP, OS:rstralex},
     and that of $\hat S_M$ is at most $n-2$ \cite{BGP}.
     We have $S_M = \hat S_M \cup \partial M^*$.

   {\rm(3)} We have a unique continuous Riemannian metric $g$ on
     $M \setminus S_M \subset M^*$
     such that the distance function induced from $g$ coincides with
     the original one of $M$ \cite{OS:rstralex}.
     The tangent space at each point in $M \setminus S_M$ is
     isometrically identified with the tangent
     cone \cite{OS:rstralex}.
     The volume measure on $M^*$ induced from $g$
     coincides with the $n$-dimensional Hausdorff measure $\Hm^n$
     \cite{OS:rstralex}.
\end{thm}

\begin{rem}
   In \cite{KMS:lap} we construct a $C^\infty$ structure only on
   $M \setminus B(S_{\delta_n},\epsilon)$.
   However this is independent of $\epsilon$ and extends to
   $M^*$.
   The $C^\infty$ structure is a refinement of the structures
   of \cite{OS:rstralex,Ot:secdiff,Pr:DC} and is
   compatible with the DC structure of \cite{Pr:DC}.
\end{rem}

Note that the metric $g$ is defined only on $M^* \setminus S_M$ and
does not continuously extend to any other point of $M$.
In general, the set of non-singular points $M^* \setminus S_M$ is not
a manifold.  There is an example of an Alexandrov space $M$ such that
$S_M$ is dense in $M$ (see \cite{OS:rstralex}).

\begin{defn}[Cut-locus]
   Let $p \in M$ be a point.
   We say that a point $x \in M$ is a \emph{cut point of $p$}
   if no minimal geodesic from $p$ contains $x$ as an interior point.
   We agree that $p$ is not a cut point of $p$.
   The set of cut points of $p$ is called the \emph{cut-locus of $p$}  
and
   denoted by $\Cut_p$.
\end{defn}

Note that $\Cut_p$ is not necessarily a closed set.
For the $W_{p,t}$ defined in \S\ref{sec:intro}, it follows that
$\bigcup_{0 < t < 1} W_{p,t} = M \setminus \Cut_p$.
The cut-locus $\Cut_p$ is a Borel subset and satisfies
$\Hm^n(\Cut_p) = 0$ \cite[Proposition 3.1]{OS:rstralex}.

By \cite[Lemma 4.1]{OS:rstralex},
the function $r_p = d(p,\cdot)$ is differentiable on
$M \setminus (S_M \cup \Cut_p \cup \{p\})$.
At any differentiable point $x$ of $r_p$,
$-\nabla r_p(x)$ is tangent to a unique minimal geodesic from $p$ to  
$x$,
where $\nabla r_p(x)$ denotes the gradient vector of $r_p$ at $x$.
This implies that the gradient vector field $\nabla r_p$ is  
continuous at
all differentiable points of $r_p$.

\section{Sobolev Spaces and
Maximum Principle}

A main purpose of this section is to prove the maximum principle
for subharmonic functions on a weighted Alexandrov space.
To prove it, we apply the maximum principle in the setting of a  
Dirichlet form
which was proved by the first named author \cite{Kw:maxprinsemi}.
For that, we need to investigate Sobolev spaces on a weighted  
Alexandrov space
by using Cheeger's theory \cite{Ch:metmeas}.
We refer \cite{FOT:Dir} for the basic terminologies of
Dirichlet forms and \cite{Ch:metmeas} for those of
Cheeger's Sobolev spaces.

Let $(X,d)$ be a proper geodesic space and $\mu$ a positive Radon
measure on $X$ with full support.
We take two open subsets $\Omega, \Omega' \subset X$ with
$\bar\Omega \subset \Omega'$, where $\bar\Omega$ is the closure of $\Omega$
in $X$.
Assume that $(X,d,\mu)$ satisfies the volume doubling condition for
metric balls contained in $\Omega'$
and a $(1,p)$-Poincar\'e inequality on any metric ball contained in $\Omega'$
for a fixed number $p > 1$ in the sense of
upper gradient (cf.~\cite{Ch:metmeas}).
$L^p(\Omega;\mu)$ denotes the Banach space of $L^p$ functions on $ 
\Omega$
with respect to $\mu$, and
$W^{1,p}(\Omega;\mu)$ the $(1,p)$-Sobolev space
of $(\Omega,d,\mu)$ defined by Cheeger \cite{Ch:metmeas}.
We denote by $g_u$ a minimal generalized upper gradient for
$u \in W^{1,p}(\Omega;\mu)$, which is unique up to modification on sets
of $\mu$-measure zero (see \cite[Theorem 2.10]{Ch:metmeas}).
Let $W^{1,p}_0(\Omega;\mu)$ be the $W^{1,p}$-closure of
the set of functions $u \in W^{1,p}(\Omega;\mu)$
such that the support of $u$ is compact and contained in $\Omega$.
Denote by $W^{1,p}_{0,\loc}(\Omega;\mu)$ the \emph{localization of
$W^{1,p}_0(\Omega;\mu)$}, i.e., the set of functions $u : \Omega \to \R$
such that, for any relatively compact open subset $U \subset X$
with $\bar U \subset \Omega$,
there is a function $u_U \in W^{1,p}_0(\Omega;\mu)$ satisfying that
$u = u_U$ $\mu$-a.e.~on $U$.
For $u \in W^{1,p}_{0,\loc}(\Omega;\mu)$,
we define $g_u : \Omega \to \R$ to be $g_{u_U}$ on each $U$.
By \cite[Corollary 2.25]{Ch:metmeas}, the function $g_u$
is defined uniquely up to modification on sets of $\mu$-measure zero.
We also call $g_u$ a minimal generalized upper gradient for $u$.

\begin{rem}
  In Cheeger's paper \cite{Ch:metmeas}, the statements of the theorems
  are described under $\Omega' = X$.  However, all the discussions in
  the proofs are local and valid for $\Omega' \subset X$.
\end{rem}

The following lemma is needed for the proof of the maximum principle.

\begin{lem} \label{lem:W-Lip}
   For a function $u : \Omega \to \R$, the following {\rm(1)} and
   {\rm (2)} are equivalent to each other.
   \begin{enumerate}
   \item We have $u \in W^{1,p}_{0,\loc}(\Omega;\mu) \cap C(\Omega)$
     and $g_u \le 1$ $\mu$-a.e., where $C(\Omega)$ denotes the set of
     continuous functions on $\Omega$.
   \item $u$ is locally $1$-Lipschitz on $\Omega$,
     i.e., for any point $x_0 \in \Omega$, there exists
     a neighborhood $B$ of $x_0$ in $\Omega$ such that
     $u|_B$ is $1$-Lipschitz.
   \end{enumerate}
\end{lem}

\begin{proof}
   The implication (2) $\implies$ (1) follows from a standard  
   discussion.

   We prove (1) $\implies$ (2).
   Assume that $u \in W^{1,p}_{0,\loc}(\Omega;\mu) \cap C(\Omega)$
   and $g_u \le 1$ $\mu$-a.e.
   We fix a point $x_0 \in \Omega$ and take a closed ball $B'$  
   centered at $x_0$
   and contained in $\Omega$.
   There is a closed ball $B$ centered at $x_0$
   such that all minimal geodesics joining two points in $B$ are
   entirely contained in $B'$.  We have $B \subset B'$ and
   $u|_{B'} \in W^{1,p}(B';\mu) \cap C(B')$.
   Denote by $C^{\Lip}(B')$ the set of Lipschitz functions on $B'$.
   By \cite[Theorem 5.3]{Ch:metmeas}
   (see also the proof of \cite[Theorem 5.1]{Ch:metmeas}), there are  
   functions
   $u_i \in C^{\Lip}(B')$ and $g_i \in L^p(B';\mu) \cap C(B')$,
   $i = 1,2,\dots$,
   such that $u_i \to u$, $g_i \to g_u$ in $L^p(B';\mu)$
   as $i \to \infty$,
   $\limsup_{i\to\infty} g_i \le g_u$ $\mu$-a.e.~on $B'$,
   and each $g_i$ is an upper gradient for $u_i$.
   Since $g_u \le 1$ $\mu$-a.e.~on $\Omega$ and since each $g_i$ is  
   continuous,
   there is a number $i_\epsilon$ for each $\epsilon > 0$ such that
   $g_i \le 1+\epsilon$ on $B'$ for all $i \ge i_\epsilon$.
   For any $x,y \in B$, we take a minimal geodesic
   $\gamma : [\,0,d(x,y)\,] \to X$ joining $x$ to $y$ with arclength
   parameter.  Since $\gamma$ is contained in $B'$, we have
   \[
   |u_i(x) - u_i(y)|
   \le \int_0^{d(x,y)} g_i\circ\gamma(s)\;ds
   \le (1+\epsilon)\,d(x,y),
   \]
   namely $u_i$ for $i \ge i_\epsilon$ is $(1+\epsilon)$-Lipschitz on  
   $B$ for any $\epsilon > 0$.
   By the Arzel\'a-Ascoli theorem, $\{u_i|_B\}$ has a subsequence which
   uniformly converges to a $1$-Lipschitz function $v$ on $B$.
   Since $u$ is continuous, we have $u = v$ on $B$ and $u$ is
   $1$-Lipschitz on $B$.
   This completes the proof.
\end{proof}

>From now on, we consider an Alexandrov space.
Let $M$ be an Alexandrov space, $\mu$ a positive Radon measure on $M$
with full support, and $\Omega \subset M$ an open subset.
We assume that $\mu$ satisfies $\BG(\kappa,N)$ on some neighborhood
$\Omega'$ of $\bar\Omega$
for two real numbers $N \ge 1$ and $\kappa$.
According to the result of Ranjbar-Motlagh \cite{Rm:Poincare},
we have a $(1,1)$-Poincar\'e inequality on any ball in $\Omega'$
in the sense of upper gradient, which implies
a $(1,p)$-Poincar\'e inequality on any ball in $\Omega'$ for any $p \ge 1$.
Since the infinitesimal Bishop-Gromov condition implies the volume
doubling condition for balls in $\Omega'$,
we can apply Cheeger's theory \cite{Ch:metmeas} of Sobolev spaces
on the metric measure space $(\Omega,d,\mu)$.

\begin{lem} \label{lem:cutmuzero}
   The set $\Cut_p \cap \Omega$ is of $\mu$-measure zero
   for any point $p \in M$.
\end{lem}

\begin{proof}
   Assume that $\mu(\Cut_p \cap \Omega) > 0$ for some point $p \in M$.
   Then, for such a point $p$,
   there is a small number $\delta > 0$ such that the set
   \[
   A := \{\;x \in \Cut_p \cap \Omega \mid d(x,\partial \Omega) > \delta,
   \ \delta < d(x,p) < 1/\delta\;\}
   \]
   has positive $\mu$-measure.  The closure of $A$ is compact and
   contained in $\Omega \setminus \{p\}$.
   Applying $\BG(\kappa,N_\Omega)$ on $\Omega$
   to the indicator function of $A$
   yields that $\mu(\Phi_{p,t}^{-1}(A)) > 0$ for any $t \in (\,0,1\,)$.
   It follows from $A \subset \Cut_p$ that
   $\Phi_{p,t}^{-1}(A) \cap \Phi_{p,t'}^{-1}(A) = \emptyset$
   for any mutually different $t,t' \in (\,0,1\,)$.
   Since $\mu$ is a Radon measure,
   this is a contradiction.
\end{proof}

\begin{lem} \label{lem:muzero}
   The set $S_M \cap \Omega$ is of $\mu$-measure zero.
\end{lem}

\begin{proof}
   We find a dense countable subset $\{p_i\}_{i=1}^\infty$ of $M$.
   Lemma \ref{lem:cutmuzero} implies that
   $\bigcup_{i=1}^\infty \Cut_{p_i} \cap \Omega$ is of $\mu$-measure  
zero.
   Thus, it suffices to prove that $S_M \subset \bigcup_{i=1}^\infty  
\Cut_{p_i}$.
   Take any point $x \in M \setminus \bigcup_{i=1}^\infty \Cut_{p_i}$.
   We are going to prove that $x$ is non-singular.
   Since $x$ is not a cut point of $p_i$, there is a minimal geodesic
   $\gamma_i$ from $p_i$ passing through $x$ for each $i$.
   Since the tangent cone $K_x$ is isometric to the magnification
   limit of $M$ around $x$, as the limit of each $\gamma_i$
   we have a straight line
   $\bar\gamma_i$ in $K_x$ passing through the vertex of $K_x$.
   Since $\{p_i\}$ is dense in $M$, the union of the images of all
   $\bar\gamma_i$'s is dense in $K_x$.  By using the splitting theorem
   of Toponogov type (cf.~\cite{Mk:line}), $K_x$ is isometric to
   $\R^n$, i.e., $x$ is non-singular.  This completes the proof.
\end{proof}

\begin{prop} \label{prop:Rad}
   Any locally Lipschitz function on $\Omega$ is differentiable
   $\mu$-a.e.~on $\Omega$.
\end{prop}

\begin{proof}
   By \cite[Theorem 10.2]{Ch:metmeas},
   any locally Lipschitz function on $\Omega$ is
   infinitesimally generalized linear $\mu$-a.e.
   By Lemma \ref{lem:muzero}, it suffices to consider
   only non-singular points.
   At a non-singular point,
   any infinitesimally generalized linear function
   is differentiable by \cite[Theorem 8.11]{Ch:metmeas}.
   This completes the proof.
\end{proof}

\begin{rem} \label{rem:Rad}
  We can improve Proposition \ref{prop:Rad} as that
  any locally Lipschitz function on $\Omega'$ is differentiable
  $\mu$-a.e.~on $\Omega'$.  This is because the proposition holds for any
  $\Omega$ with $\bar\Omega \subset \Omega'$ and there is an
  increasing sequence $\Omega_i$, $i=1,2,\dots$, such that
  $\bar\Omega_i \subset \Omega$ and $\bigcup_i \Omega_i = \Omega$.
\end{rem}

With the help of Proposition \ref{prop:Rad}, we define a Dirichlet form.
Denote by $C_0^{\Lip}(\Omega)$ the set of Lipschitz functions
with compact support in $\Omega$.
We define a bilinear form by
\[
\mathcal{E}^\mu(u,v) := \int_\Omega \langle\nabla u,\nabla v\rangle  
\; d\mu,
\quad u,v \in C_0^{\Lip}(\Omega),
\]
where $\langle\cdot,\cdot\rangle$ is the Riemannian metric on
$M \setminus S_M$.
Note that $\langle\nabla u,\nabla v\rangle$ is $\mu$-a.e.~defined
by Proposition \ref{prop:Rad}.

\begin{lem} \label{lem:Dir}
   The bilinear form $(\mathcal{E}^\mu,C_0^{\Lip}(\Omega))$ is  
   closable in $L^2(\Omega;\mu)$
   and its closure,
   say $(\mathcal{E}^\mu,W^{1,2}_0(\Omega;\mu))$, coincides with  
Cheeger's
   $(1,2)$-Sobolev space with the Dirichlet boundary condition
   and is a strongly local
   regular Dirichlet form on $L^2(\Omega;\mu)$ in the sense of
   \cite{FOT:Dir}.
\end{lem}

\begin{proof}
   By \cite[Theorem 5.1]{Ch:metmeas}, the minimal upper gradient
   of a locally Lipschitz function $u : \Omega \to \R$ coincides with
   the local Lipschitz constant of $u$, which is equal to $\|\nabla u\|$
   at any differentiable point of $u$, where
   $\|\cdot\|$ denotes the norm induced from the Riemannian metric on
   $M \setminus S_M$.
   Therefore, by Proposition \ref{prop:Rad},
   the Cheeger's energy of $u \in C_0^{\Lip}(\Omega)$ coincides with
   $\mathcal{E}^\mu(u,u)$, so that, by \cite[Theorem 4.24]{Ch:metmeas},
   $(\mathcal{E}^\mu,C_0^{\Lip}(\Omega))$ is closable
   and its closure $(\mathcal{E}^\mu,W^{1,2}_0(\Omega;\mu))$
   coincides with Cheeger's Sobolev space with Dirichlet boundary  
condition.
   The strong locality, the regularity and the Markovian property
   of $(\mathcal{E}^\mu,W^{1,2}_0(\Omega;\mu))$ are all obvious.
   This completes the proof.
\end{proof}

Let $\rho_\Omega$ denote the \emph{intrinsic metric on $\Omega$  
induced from
the Dirichlet form $(\mathcal{E}^\mu,W^{1,2}_0(\Omega;\mu))$}
(cf.~\cite{BM:Saint}), i.e., for $x,y \in \Omega$,
\[
\rho_\Omega(x,y) := \sup\{\; u(x) - u(y) \mid
u \in W^{1,2}_{0,\loc}(\Omega;\mu) \cap C(\Omega),\ d\mu_{\langle u 
\rangle} \le d\mu
\;\},
\]
where $\mu_{\langle u\rangle}$ is the energy measure of $u$
(cf.~\cite[\S 3.2]{FOT:Dir}).
It is known that $\rho_\Omega$ is a pseudo-metric in general.

\begin{prop} \label{prop:rho-d}
   For any $x_0 \in \Omega$, there exists a neighborhood $B$ of $x_0$
   in $\Omega$ such that $\rho_\Omega = d$ on $B \times B$.
   In particular, $\rho_\Omega$ is a metric which induces
   the same topology of $d$ on $\Omega$.
\end{prop}

\begin{proof}
   It is easy to prove that $\mu_{\langle u\rangle} = g_u^2\,d\mu$  
for any
   $u \in W^{1,2}_{0,\loc}(\Omega;\mu)$, so that
   $d\mu_{\langle u\rangle} \le d\mu$ is equivalent to $g_u \le 1$
   $\mu$-a.e.
   By Lemma \ref{lem:W-Lip} we have
   \[
   \rho_\Omega(x,y) = \sup\{\; u(x) - u(y) \mid
   \text{$u : \Omega \to \R$ is locally $1$-Lipschitz}\;\}.
   \]

   Let us prove that $d(x,y) \le \rho_\Omega(x,y)$ for any $x,y \in  
\Omega$.
   We fix $x,y \in \Omega$ and set $u(z) := d(y,z)$, $z \in \Omega$.
   Then, $u$ is $1$-Lipschitz and $u(x) - u(y) = d(x,y)$,
   which imply $d(x,y) \le \rho_\Omega(x,y)$.

   Fix a point $x_0 \in \Omega$.
   There is a neighborhood $B$ of $x_0$ in $\Omega$
   such that all minimal geodesics joining two points in $B$ are
   entirely contained in $\Omega$.
   Let us prove that $d(x,y) \ge \rho_\Omega(x,y)$ for any $x,y \in B$.
   Take any locally $1$-Lipschitz function $u : \Omega \to \R$.
   It follows from the condition for $B$ that
   $u$ is (globally) $1$-Lipschitz on $B$, so that, for any $x,y \in B$,
   we have $u(x) - u(y) \le d(x,y)$, which implies
   $\rho_\Omega(x,y) \le d(x,y)$.
   This completes the proof.
\end{proof}

\begin{defn}[$\mu$-subharmonicity] \label{defn:subh}
   A function $u \in W^{1,2}_{0,\loc}(\Omega;\mu)$ is said to be
   \emph{$\mu$-subharmonic} if
   \[
   \int_\Omega \langle \nabla u,\nabla f\rangle \; d\mu \le 0
   \]
   for any nonnegative function $f \in C_0^{\Lip}(\Omega)$.
\end{defn}

Using Lemma \ref{lem:Dir} and Proposition \ref{prop:rho-d},
we prove the maximum principle.

\begin{thm}[Maximum Principle] \label{thm:maxprin}
  Assume that $\Omega$ is connected.
  If a continuous $\mu$-subharmonic function
  $u \in W^{1,2}_{0,\loc}(\Omega;\mu)$
  attains its maximum in $\Omega$, then $u$ is constant on $\Omega$.
\end{thm}

\begin{proof}
  As is mentioned before, we have
  the volume doubling condition for balls in $\Omega'$
  and a $(1,p)$-Poincar\'e inequality on any ball in $\Omega'$
  for any $p \ge 1$.
  By \cite[Theorem 5.1 and Corollary 9.8]{HK:SobPoin},
  we have a $(2,2)$-Poincar\'e inequality on $\Omega$,
  which together with Proposition \ref{prop:rho-d} implies
  a parabolic Harnack inequality on $\Omega$
  (see \cite[Theorem 3.5]{St:DirIII}).
  Therefore, the same proof as in \cite[Theorem 7.4]{St:heat} works to obtain
  a strictly positive locally H\"older continuous heat kernel
  $p_{\Omega}(t,x,y)$, $(t,x,y)\in(0,\, \infty)\times \Omega\times\Omega$,
  associated to $(\mathcal{E}^\mu,W^{1,2}_0 (\Omega;\mu))$ on
  $L^2(\Omega;\mu)$.
  The method of the proof of  \cite[Proposition 7.5]{St:heat}
  also works to obtain the strong Feller property of the semigroup
  $T_t^{\Omega}f(x):=\int_{\Omega}p_ {\Omega}(t,x,y)f(y)\,\mu(dy)$,
  where $f$ is a bounded Borel function on $\Omega$.
  Owing to the strong maximum principle due to the
  first named author \cite[Theorems 1.3 and 8.5 with the last remark
  in Example 8.2]{Kw:maxprinsemi}, we have the
  theorem.
\end{proof}

\section{Laplacian Comparison Theorem}
\label{sec:LapComp}

The purpose of this section is to prove the following theorem.
We set $\cot_\kappa(r) := s_\kappa'(r)/s_\kappa(r)$ for the function
$s_\kappa$ defined in \S\ref{sec:intro}.

\begin{thm}[Laplacian Comparison Theorem] \label{thm:LapComp}
   Let $M$ be an Alexandrov space, $\mu$ a positive Radon measure on $M$
   with full support, and $\Omega \subset M$ an open subset.
   If $\mu$ satisfies $\BG(\kappa,N_\Omega)$ on $\Omega$
   for two real numbers $N_\Omega \ge 1$ and $\kappa$, then
   we have
   \begin{equation}
     \label{eq:LapComp}
     \int_M \langle\nabla r_p,\nabla f\rangle \; d\mu
     \ge
     \int_M \{ -(N_\Omega-1)\cot_\kappa \circ r_p \}\, f \;d\mu
   \end{equation}
   for any point $p \in M$ and
   for any nonnegative function $f \in C_0^{\Lip}(\Omega \setminus \{p\})$.
\end{thm}

Note that $r_p$ and $f$ above are differentiable $\mu$-a.e.~on $\Omega$
by Proposition \ref{prop:Rad} and Remark \ref{rem:Rad}.

Let $V$ be a function on $M$ with a certain regularity condition.
For the measure $d\mu(x) = e^{-V(x)}\, d\Hm^n(x)$, we define
\[
\Delta_\mu := \Delta + \nabla V = -e^{V}\divv(e^{-V}\nabla\cdot),
\]
where $\Delta$ is the nonnegative Laplacian and $\nabla V$ is the
gradient vector field of $V$, considered to be a directional
derivative.  The inequality \eqref{eq:LapComp} is a weak form of the  
formal
inequality
\begin{equation}
   \label{eq:LapComp-ptw}
   \Delta_\mu r_p \ge -(N_\Omega-1)\cot_\kappa \circ r_p
\end{equation}
on $\Omega \setminus \{p\}$, and $-(N_\Omega-1)\cot_\kappa \circ
r_p$ is the Laplacian of the distance function on an
$N_\Omega$-dimensional complete simply connected space form of
constant curvature $\kappa$, provided that $N_\Omega$ is an integral
number with $N_\Omega \ge 2$.
We do not know if the pointwise inequality \eqref{eq:LapComp-ptw}
implies the weak form \eqref{eq:LapComp} in general.
However, if $M$ is a singular Riemannian space and if $V$ is a $C^2$
function, then \eqref{eq:LapComp-ptw} implies \eqref{eq:LapComp}.
For the proof of this, we first prove \eqref{eq:LapComp-ptw}
in the sense of barrier by the same way as in \cite{EH:split}
and then prove \eqref{eq:LapComp}.  The details are omitted here.

Since, for an $n$-dimensional Alexandrov space of
curvature $\ge \kappa$, the $n$-dimensional Hausdorff measure
$\Hm^n$ satisfies $\BG(\kappa,n)$ (see \cite{KwSy:BG-Alex}),
the above Laplacian Comparison Theorem (Theorem \ref{thm:LapComp})
leads us to the following.

\begin{cor} \label{cor:LapComp}
   If $M$ is an $n$-dimensional Alexandrov space of
   curvature $\ge \kappa$, then for any $p \in M$ we have
   $\Delta r_p \ge -(n-1)\cot_\kappa \circ r_p$ on $M\setminus \{p\}$
   in the weak sense.
\end{cor}

Since the Riemannian metric on an Alexandrov space is not continuous
on any singular point, a standard proof of the Laplacian comparison  
theorem
for Riemannian manifolds does not work.
Renesse \cite{Rs:compalex} proved Corollary \ref{cor:LapComp}
under some additional condition.
In the case of $\mu = \Hm^n$ with $\BG(\kappa,n)$,
another different proof using a version of Green formula can be seen in
our previous paper \cite{KwSy:lapcomp}.

\begin{proof}[Proof of Theorem \ref{thm:LapComp}]
  Let $f \in C_0^{\Lip}(\Omega \setminus \{p\})$ be any
  nonnegative function.
  By Proposition \ref{prop:Rad} and Remark \ref{rem:Rad},
  $f$ and $r$ are differentiable $\mu$-a.e.~on $\Omega$.
  It follows from $t\,r_p(\Phi_{p,t}(x)) = r_p(x)$ that,
  for $\mu$-almost all $x \in \Omega$,
  \[
  \frac{d}{dt} \Phi_{p,t}(x)\Bigr|_{t=1} = -r_p(x) \nabla r_p(x)
  \]
  and so
  \begin{align*}
    \langle r_p\nabla r_p, \nabla f \rangle
    = -\left\langle \der{}{t}\Phi_{p,t}(x)\Bigr|_{t=1},\nabla f\right 
    \rangle
    = -\der{}{t} f\circ\Phi_{p,t}(x)\Bigr|_{t=1}.
  \end{align*}
  For $0 < t < 1$ we define a function $F_t : M \to \R$ by
  \[
  F_t(x) :=
  \begin{cases}
    (f\circ\Phi_{p,t}(x)-f(x))/(1-t) &\text{for $x \in W_{p,t}$},\\
    0 &\text{for $x \in M \setminus W_{p,t}$}.
  \end{cases}
  \]
  Then we have
  \[
  \lim_{t\to 1-0} F_t(x) = -\der{}{t}f\circ\Phi_{p,t}(x)\Bigr|_{t=1}
  \]
  for $\mu$-almost all $x \in \Omega$.
  It follows from $d(x,\Phi_{p,t}(x)) = (1-t)r_p(x)/t$,
  $x \in W_{p,t}$, that $|F_t(x)| \le L\,r_p(x)/t$
  for all $t \in (\,0,1\,)$ and $x \in M$,
  where $L$ is a Lipschitz constant of $f$.
  Thus, the dominated convergence theorem implies
  \begin{align*}
    &\int_M \langle r_p\nabla r_p, \nabla f \rangle \; d\mu
    = -\int_M \der{}{t} f\circ\Phi_{p,t}(x)\Bigr|_{t=1}\;d\mu(x)\\
    &= \lim_{t\to 1-0} \int_M F_t(x) \; d\mu(x)\\
    &= \lim_{t\to 1-0} \left[
      \int_{W_{p,t}} \frac{f\circ\Phi_{p,t}(x)}{1-t}\;d\mu(x)
      -\int_{W_{p,t}} \frac{f(x)}{1-t}\;d\mu(x) \right]\\
    \intertext{and by $\BG(\kappa,N_\Omega)$ and $f(x)/(1-t) \ge 0$,}
    &\ge \liminf_{t \to 1-0} \Bigl[
    \int_M \frac{t\,s_\kappa(t\,r_p(x))^{N_\Omega-1} f(x)}
    {(1-t)s_\kappa(r_p(x))^{N_\Omega-1}}\; d\mu(x)
    -\int_M \frac{f(x)}{1-t}\;d\mu(x) \Bigr]\\
    &\ge \int_M \liminf_{t\to 1-0}
    \frac{t\,s_\kappa(t\,r_p(x))^{N_\Omega-1}
      - s_\kappa(r_p(x))^{N_\Omega-1}}
    {(1-t)s_\kappa(r_p(x))^{N_\Omega-1}}\,f(x)\;d\mu(x)\\
    &= -\int_M \der{}{t} \{ t\,s_\kappa(t\,r_p(x))^{N_\Omega-1} \} 
    \Bigr|_{t=1}\;
    \frac{f(x)}{s_\kappa(r_p(x))^{N_\Omega-1}}\;d\mu(x)\\
    &= \int_M \{-1-(N_\Omega-1)r_p(x)\cot_\kappa(r_p(x))\}\,f(x) \; d 
    \mu(x).
  \end{align*}
  By $\nabla(r_pf) = f\nabla r_p + r_p\nabla f$,
  we see that
  \[
  \langle r_p\nabla r_p,\nabla f\rangle = \langle\nabla r_p,\nabla 
  (r_pf)\rangle
  -f \ \ \text{$\mu$-a.e.}
  \]
  and therefore,
  \begin{equation}
    \label{eq:rf}
    \int_M \langle\nabla r_p,\nabla(r_pf)\rangle \; d\mu
    \ge \int_M \{ -(N_\Omega-1)\cot_\kappa(r_p(x))\}\, r_p(x) f(x)  
    \; d\mu(x).
  \end{equation}
  We now give any nonnegative function
  $\hat f \in C_0^{\Lip}(\Omega \setminus \{p\})$.
  Set
  $f(x) := \hat f(x)/r_p(x)$ for $x \neq p$ and $f(p) := 0$.
  Then, $f : M \to \R$ is a nonnegative function which belongs to
  $C_0^{\Lip}(\Omega \setminus \{p\})$.
  \eqref{eq:rf} implies \eqref{eq:LapComp} for $\hat f$.
  This completes the proof.
\end{proof}

In our paper \cite{KMS:lap},
we proved for an Alexandrov space $M$
the existence of the heat kernel of $M$ and the discreteness of
the spectrum of the Laplacian (the generator of the Dirichlet energy  
form)
on a relatively compact domain in $M$.
As applications to Theorem \ref{thm:LapComp},
we have the following heat kernel
and first eigenvalue comparison results,
which generalize the results of Cheeger-Yau \cite{CY:heatkernel} and
Cheng \cite{Ch:eigencomp}.

$B(p,r)$ denotes the metric ball centered at $p$ and of radius $r$
and $M^n(\kappa)$ an $n$-dimensional
complete simply connected space form of curvature $\kappa$.

\begin{cor} \label{cor:heatcomp}
  Let $M$ be an $n$-dimensional Alexandrov space and assume that
  $\Hm^n$ satisfies $\BG(\kappa,n)$.
  Let $\Omega \subset M$ be an open subset containing $B(p,r)$ for
  a number $r > 0$.
  Denote by $h_t : \Omega \times \Omega \to \R$, $t > 0$,
  the heat kernel on $\Omega$ with Dirichlet boundary condition,
  and by $\bar h_t : B(\bar p,r) \times B(\bar p,r) \to \R$
  that on $B(\bar p,r)$
  for a point $\bar p \in M^n(\kappa)$.
  Then, for any $t > 0$ and $q \in B(p,r)$, we have
  \[
  h_t(p,q) \ge \bar h_t(\bar p,\bar q),
  \]
  where $\bar q \in M^n(\kappa)$ is a point such that
  $d(\bar p,\bar q) = d(p,q)$.
\end{cor}

\begin{cor} \label{cor:eigen}
  Let $M$ be an $n$-dimensional Alexandrov space and $r > 0$ a real
  number.  Assume that $\Hm^n$ satisfies
  $\BG(\kappa,n)$.
  Denote by $\lambda_1(B(p,r))$ the first eigenvalue of
  the Laplacian on $B(p,r)$ with Dirichlet boundary condition,
  and by $\lambda_1(B(\bar p,r))$ that on $B(\bar p,r)$ for
  a point $\bar p \in M^n(\kappa)$.
  Then we have
  \[
  \lambda_1(B(p,r)) \le \lambda_1(B(\bar p,r)).
  \]
\end{cor}

Once we have Theorem \ref{thm:LapComp},
the proofs of Corollaries \ref{cor:heatcomp} and \ref{cor:eigen}
are the same as of Renesse
\cite[Theorem II and Corollary 1]{Rs:compalex}.
We verify that the local $(L^1,1)$-volume regularity
is not needed in the proof of \cite[Theorem II]{Rs:compalex}.
We also obtain a Brownian motion comparison theorem
in the same way as in \cite{Rs:compalex}.

\section{Splitting Theorem} \label{sec:splitting}

We prove the Topological Splitting Theorem (Theorem \ref{thm:splitting})
following the idea of Cheeger-Gromoll \cite{CG:split}.
However, we still need some extra lemmas to fit the discussions of
\cite{CG:split} to Alexandrov spaces.

Let $M$ be a non-compact Alexandrov space and
$\gamma$ a \emph{ray} in $M$, i.e., a geodesic defined on $[\,0,+ 
\infty\,)$
such that $d(\gamma(s),\gamma(t)) = |s-t|$ for any $s,t \ge 0$.

\begin{defn}[Busemann function]
  The \emph{Busemann function $b_\gamma : M \to \R$ for $\gamma$}
  is defined by
  \[
  b_\gamma(x) := \lim_{t \to +\infty} \{ t - d(x,\gamma(t)) \},
  \quad x \in M.
  \]
\end{defn}

It follows from the triangle inequality that
$t - d(x,\gamma(t))$ is monotone non-decreasing in $t$,
so that the limit above exists.
$b_\gamma$ is a $1$-Lipschitz function.

\begin{defn}[Asymptotic relation]
  We say that a ray $\sigma$ in $M$ is \emph{asymptotic to $\gamma$}
  if there exist a sequence $t_i \to +\infty$, $i = 1,2,\dots$,
  and minimal geodesics
  $\sigma_i : [\,0,l_i\,] \to M$ with $\sigma_i(l_i) = \gamma(t_i)$
  such that $\sigma_i$ converges to $\sigma$ as $i \to \infty$,
  (i.e., $\sigma_i(t) \to \sigma(t)$ for each $t$).
\end{defn}

For any point in $M$, there is a ray asymptotic to $\gamma$
from the point.
Any subray of a ray asymptotic to $\gamma$ is asymptotic to $\gamma$.
By the same proof as for Riemannian manifolds
(cf.~\cite[Theorem 3.8.2(3)]{SST:totcurv}), for any ray $\sigma$  
asymptotic
to $\gamma$, we have
\begin{equation}
  \label{eq:blinear}
  b_\gamma \circ \sigma(s) = s + b_\gamma \circ \sigma(0)
  \quad\text{for any $s \ge 0$.}
\end{equation}

\begin{lem} \label{lem:diff}
  Let $f : M \to \R$ be a $1$-Lipschitz function and
  $u,v \in \Sigma_pM$ two directions at a point $p \in M$.
  If the directional derivative of $f$ to $u$ is equal to $1$
  and that to $v$ equal to $-1$,
  then the angle between $u$ and $v$ is equal to $\pi$.
\end{lem}

\begin{proof}
  There are points $x_t, y_t \in M$, $t > 0$, such that
  $d(p,x_t) = d(p,y_t) = t$ for all $t > 0$ and that
  the direction at $p$ of $px_t$ (resp.~$py_t$)
  converges to $u$ (resp.~$v$) as $t \to 0$.
  The assumption for $f$ tells us that
  \[
  \lim_{t \to 0} \frac{f(x_t) - f(p)}{t} = 1 \quad\text{and}\quad
  \lim_{t \to 0} \frac{f(y_t) - f(p)}{t} = -1,
  \]
  which imply
  \[
  \lim_{t \to 0} \frac{d(x_t,y_t)}t \ge
  \lim_{t \to 0} \frac{f(x_t) - f(y_t)}{t} = 2.
  \]
  This completes the proof.
\end{proof}

\begin{lem} \label{lem:bfunc}
  Assume that a ray $\sigma : [\,0,\,+\infty\,) \to M$ is asymptotic
  to a ray $\gamma : [\,0,\,+\infty\,) \to M$, and let $s$ be a given
  positive number.
  
  {\rm(1)} If $\sigma(s)$ is a non-singular point, then
  $b_\gamma$ is differentiable at $\sigma(s)$ and
  $\nabla b_\gamma(\sigma(s))$ is tangent to $\sigma$.
  
  {\rm(2)} Among all rays emanating from $\sigma(s)$,
  only the subray $\sigma|_{[\,s,+\infty\,)}$ of $\sigma$
  is asymptotic to $\gamma$.
\end{lem}

\begin{proof}
  (1) follows from the same discussion as for Riemannian manifolds
  (see \cite[Theorem 3.8.2]{SST:totcurv}), in which
  we need the total differentiability
  of the distance function from a compact subset of $M$.
  This is obtained in the same way as in
  \cite[Theorem 3.5 and Lemma 4.1]{OS:rstralex}.
  
  (2): Take any ray $\tau$ from $\sigma(s)$ asymptotic to $\gamma$.
  By \eqref{eq:blinear}, the derivative of $b_\gamma\circ\tau$ is
  equal to $1$.  Therefore, using Lemma \ref{lem:diff} yields
  that the angle between $\sigma|_{[0,s]}$ and $\tau$ is equal to $ 
  \pi$,
  so that $\sigma'(s) = \tau'(0)$.
  This completes the proof.
\end{proof}

Note that if $\sigma(s)$ is a non-singular point, then
Lemma \ref{lem:bfunc}(1) implies (2).

\begin{lem} \label{lem:bsplitting}
  Let $\gamma$ be a straight line in $M$.
  Denote by $b_+$ the Busemann function for
  $\gamma_+ := \gamma|_{[\,0,+\infty\,)}$ and by $b_-$
  that for $\gamma_- := \gamma|_{(\,-\infty,0\,]}$.
  If $b_+ + b_- \equiv 0$ holds, then
  $M$ is covered by disjoint straight lines bi-asymptotic to $\gamma$.
  In particular, $b_+^{-1}(t)$ for all $t \in \R$ are homeomorphic
  to each other and $M$ is homeomorphic to $b_+^{-1}(t) \times \R$.
\end{lem}

\begin{proof}
  Take any point $p \in M$ and a ray $\sigma : [\,0,+\infty\,) \to M$
  from $p$ asymptotic to $\gamma_+$.
  For any $s > 0$, the directional derivatives of $b_+$ to the two
  opposite directions at $\sigma(s)$ tangent to $\sigma$ are
  $-1$ and $1$, respectively.
  Since $b_- = -b_+$ and by Lemma \ref{lem:diff},
  a ray from $\sigma(s)$ asymptotic to $\gamma_-$ is unique
  and contains $\sigma([\,0,s\,])$.
  By the arbitrariness of $s > 0$,
  $\sigma$ extends to a straight line bi-asymptotic to $\gamma$.
  Namely, for a given point $p \in M$, we have a straight line
  $\sigma_p$ passing through $p$ and bi-asymptotic to $\gamma$.
  By Lemma \ref{lem:bfunc}(2), any ray from a point in $\sigma_p$
  asymptotic to $\gamma_\pm$ is a subray of $\sigma_p$.
  In particular, $\sigma_p$ is unique (up to parameters) for a given  
  $p$,
  and for any two points $p,q \in M$
  the images of $\sigma_p$ and $\sigma_q$ either coincide or
  do not intersect each other.
  $M$ is covered by $\{\sigma_p\}_{p \in M}$ and
  this completes the proof.
\end{proof}

\begin{lem} \label{lem:bsubh}
  Let $\mu$ be a positive Radon measure on $M$ with full support
  and let $\Omega \subset M$ be an open subset.
  Assume that $\mu$ satisfies $\BG(0,N_\Omega)$ on $\Omega$
  for a real number $N_\Omega \ge 1$.
  Then, the Busemann function $b_\gamma$ for any ray $\gamma$ in $M$
  is $\mu$-subharmonic on $\Omega$
  in the sense of Definition {\rm\ref{defn:subh}}.
\end{lem}

\begin{proof}
  We take a sequence $t_i \to +\infty$, $i=1,2,\dots$.
  Since $r_{\gamma(t_i)}$, $b_\gamma$ are both Lipschitz,
  they are
  differentiable $\mu$-a.e.~on $\Omega$ by Proposition \ref{prop:Rad}
  and Remark \ref{rem:Rad}.
  Let $x \in \Omega$ be any non-singular point where $r_{\gamma(t_i)}$  
  and $b_\gamma$ are all differentiable.
  We have a unique minimal geodesic $\sigma_{x,i}$ from $x$ to
  $\gamma(t_i)$ and $-\nabla r_{\gamma(t_i)}(x)$ is tangent to it.
  A ray $\sigma_x$ from $x$ asymptotic to $\gamma$ is unique
  and $\nabla b_\gamma(x)$ is tangent to it.
  Since $\sigma_{x,i} \to \sigma_x$ as $i \to \infty$,
  we have $-\nabla r_{\gamma(t_i)}(x) \to \nabla b_\gamma(x)$.
  For any nonnegative function $f \in C_0^{\Lip}(\Omega)$,
  the dominated convergence theorem
  and Laplacian Comparison Theorem (Theorem \ref{thm:LapComp})
  show that
  \begin{align*}
    \int_\Omega \langle \nabla b_\gamma,\nabla f\rangle \; d\mu
    &= -\lim_{i\to\infty} \int_\Omega \langle \nabla r_{\gamma(t_i)}, 
    \nabla f\rangle
    \; d\mu\\
    &\le (N_\Omega-1) \lim_{i\to\infty} \int_{\supp f}
    \frac{f}{r_{\gamma(t_i)}} \; d\mu = 0,
  \end{align*}
  where $\supp f$ is the support of $f$.
  This completes the proof.
\end{proof}

\begin{proof}[Proof of Theorem \ref{thm:splitting}]
  Let $\Omega \subset M$ be any connected, relatively compact and open
  subset.
  The assumption of the theorem yields that
  there is a number $N_\Omega \ge 1$
  such that $\mu$ satisfies $\BG(0,N_\Omega)$ on a
  neighborhood of $\bar\Omega$.
  By Lemma \ref{lem:bsubh},
  $b := b_+ + b_-$ is $\mu$-subharmonic on $\Omega$, where
  $b_+$ and $b_-$ are the Busemann functions as in
  Lemma \ref{lem:bsplitting} for a straight line $\gamma$ in $M$.
  It follows from the triangle inequality that $b \le 0$.
  We have $b \circ \gamma \equiv 0$ by the definition of $b_\pm$.
  The maximum principle (Theorem \ref{thm:maxprin}) proves that
  $b \equiv 0$ on $\Omega$ if $\Omega$ intersects $\gamma$.
  By the aribitrariness of such $\Omega$,
  we have $b \equiv 0$ on $M$.
  Lemma \ref{lem:bsplitting} proves the theorem.
\end{proof}

\begin{proof}[Proof of Corollaries \ref{cor:splitting}
  and \ref{cor:splitting-BE}]
  Let $M$, $N$ and $V$ be as in Corollary \ref{cor:splitting-BE}.
  In Corollary \ref{cor:splitting}, we assume that
  $N = n$ and $V = 1$, in which case we have $\Ric_{N,\mu} = \Ric$
  and $\Delta_\mu = \Delta + \nabla V = \Delta$.
  For the corollaries, it suffices to prove
  that if $M$ contains a straight line, then $M$ is isometric to $M'
  \times \R$ and $V$ is constant on $\{x\} \times \R$ for each $x \in
  M'$.  This is because if $M' \times \R$ is a singular Riemannian
  space, then so is $M'$.
  
  We first assume that $N < +\infty$.
  Since any geodesic joining two points in
  $M \setminus S_M$ is contained in $M \setminus S_M$
  (see Lemma \ref{lem:geod}),
  the condition $\Ric_{N,\mu} \ge 0$ on $M \setminus S_M$ implies
  $\BG(0,N)$ for $\mu$ on $M \setminus S_M$
  (see \cite{BQ:volcomp} and also \cite{St:geom,LV:Ricmm}).
  By $\Hm^n(S_M) = 0$,
  $\mu$ satisfies $\BG(0,N)$ on $M$ (an easy discussion proves that
  for any convergent sequence $p_i \to p_\infty$ in $M$, $\BG(0,N)$
  for $p = p_i$ implies $\BG(0,N)$ for $p = p_\infty$).
  We then apply Theorem \ref{thm:splitting} to $M$ and $\mu$ for
  $N_\Omega = N$.
  In the proof of the theorem, we obtain that $b_+$ and $b_-$ are both
  $\mu$-subharmonic and $b_+ + b_- = 0$.
  Therefore, $b_{\pm}$ is $\mu$-harmonic, i.e.,
  a weak solution of $\Delta_\mu b_{\pm} = 0$ on $M \setminus S_M$.
  By the regularity theorem of elliptic differential equations,
  $b_+$ is of $C^2$ on $M \setminus S_M$ and satisfies
  $\Delta_\mu b_+ = 0$ pointwise on $M \setminus S_M$.
  We use the generalized Weitzenb\"ock formula for $\Ric_{N,\mu}$:
  \begin{align*}
    & -\Delta_\mu\Bigl(\frac{\|\nabla f\|^2}{2}\Bigr)
    + \langle\nabla \Delta_\mu f,\nabla f\rangle\\
    &= \frac{(\Delta_\mu f)^2}{N} + \Ric_{N,\mu}(\nabla f,\nabla f)
    + \Bigl\| \Hess f + \frac{\Delta f}{n}I_n \Bigr\|_{HS}^2\\
    &\quad + \frac{n}{N(N-n)} \Bigl( -\frac{N-n}{n}\Delta f
    + \langle\nabla V,\nabla f\rangle \Bigr)^2
  \end{align*}
  for any $C^2$ function $f : M \setminus S_M \to \R$
  (see \cite[(14.46)]{V:oldnew}),
  where $I_n$ denotes the identity operator and $\|\cdot\|_{HS}$ the
  Hilbert-Schmidt norm.
  Since $\|\nabla b_+\| = 1$ and
  $\Ric_{N,\mu}(\nabla b_+,\nabla b_+) \ge 0$,
  putting $f := b_+$ in the above formula yields that
  $\Hess b_+ = -(\Delta b_+/n)I_n$ and
  $((N-n)/n)\Delta b_+ =
  \langle\nabla V,\nabla b_+\rangle$.
  Since
  \[
  0 = \Delta_\mu b_+ = \Delta b_+ + \langle\nabla V,\nabla b_+ \rangle
  = \frac{N}{n}\Delta b_+,
  \]
  we have $\Hess b_+ = 0$ and $\langle\nabla V,\nabla b_+\rangle = 0$
  on $M \setminus S_M$.
  Thus, $b_+$ is a linear function along any geodesic in $M  
  \setminus S_M$.
  Since any geodesic segments in $M$ can be approximated by
  geodesic segments in $M \setminus S_M$,
  $b_+$ is linear along any geodesic in $M$.
  Since $M$ is covered by straight lines bi-asymptotic to $\gamma$,
  $b_+$ is averaged $D^2$ in the sense of \cite{Mk:splitting}.
  The isometric splitting follows from
  \cite[Theorem A]{Mk:splitting}.
  Since $\langle\nabla V,\nabla b_+\rangle = 0$ on $M \setminus S_M$,
  $V$ is constant along each straight line bi-asymptotic to
  $\gamma$.  This proves the corollaries in the case of $N < +\infty$.
  
  We next consider the case where $N = +\infty$.
  By $\Ric_{\infty,\mu} \ge 0$, the same discussion as in \cite 
  {Lch:BE},
  \cite[(1)]{FLZ:splitting} and \cite[(2.21)]{WW:comp-BE}
  leads to
  \begin{align*}
    \Delta_\mu r_p(x) &\ge -\frac{n-1}{r_p(x)} + \frac{2V(x)}{r_p(x)}
    -\frac{2}{r_p(x)^2} \int_0^{r_p(x)} V(\gamma(s))\;ds\\
    &\ge -\frac{n-1+2\sup_M V - 2V(x)}{r_p(x)}
  \end{align*}
  for any $p \in M$ and $x \in M \setminus (S_M \cup \Cut_p \cup
  \{p\})$,
  where $\gamma$ is a unique unit speed geodesic joining $p$ to $x$.
  Therefore, for a given compact subset $\Omega \subset M$,
  setting $N_\Omega := n+2\sup_M V - 2\inf_\Omega V$,
  we have
  \begin{align*}
    \Delta_\mu r_p \ge -\frac{N_\Omega-1}{r_p}
  \end{align*}
  on $\Omega \setminus (S_M \cup \Cut_p \cup \{p\})$, which together  
  with
  a standard discussion (cf.~the proof of \cite[Theorem 3.2]{Ota:mcp})
  yields $\BG(0,N_\Omega)$ on $\Omega$ for $\mu$.
  (Also, we directly obtain the Laplacian comparison
  (Theorem \ref{thm:LapComp}) as was stated under \eqref{eq:LapComp-ptw}).
  By applying Theorem \ref{thm:splitting},
  $M$ splits as $M' \times \R$ homeomorphically.
  We have $\Delta_\mu b_+ = 0$ on $M \setminus S_M$.
  Apply the generalized Weitzenb\"ock formula for $\Ric_{\infty,\mu}$:
  \[
  -\Delta_\mu\Bigl(\frac{\|\nabla f\|^2}{2}\Bigr) + \langle\nabla  
  \Delta_\mu f,\nabla f\rangle
  = \Ric_{\infty,\mu}(\nabla f,\nabla f) + \| \Hess f \|_{HS}^2.
  \]
  By setting $f := b_+$, the left-hand side vanishes, so that,
  by $\Ric_{\infty,\mu} \ge 0$, we have
  $\Ric_{\infty,\mu}(\nabla b_+,\nabla b_+) = 0$ and $\Hess b_+ = 0$
  on $M \setminus S_M$.
  In the same way as in the case of $N < +\infty$,
  we obtain that $M$ is isometric to $M' \times \R$.
  Since $\Ric(\nabla b_+,\nabla b_+) = 0$,
  we have
  \[
  0 = \Ric_{\infty,\mu}(\nabla b_+,\nabla b_+)
  = \Hess V(\nabla b_+,\nabla b_+)
  = \frac{\partial^2}{\partial t^2}V
  \]
  in the coordinate $(x,t) \in M' \times \R = M$, so that
  $V$ is linear along each line $\{x\} \times \R$, $x \in M'$.
  Since $V$ is bounded,
  it is constant along each line $\{x\} \times \R$, $x \in M'$.
\end{proof}


\def\cprime{$'$}
\providecommand{\bysame}{\leavevmode\hbox to3em{\hrulefill}\thinspace}
\providecommand{\MR}{\relax\ifhmode\unskip\space\fi MR }
\providecommand{\MRhref}[2]{%
  \href{http://www.ams.org/mathscinet-getitem?mr=#1}{#2}
}
\providecommand{\href}[2]{#2}

\end{document}